\documentclass[reqno,tbtags,intlimits,a4paper,oneside,12pt]{amsart}
\usepackage[cp1251]{inputenc}%
\usepackage[english]{babel}
\usepackage{amssymb,upref,calc,url}
\usepackage[mathscr]{eucal}
\usepackage{graphicx}
\usepackage{amsmath,amscd,amsthm,verbatim}
\usepackage[all]{xy}
\usepackage[in]{fullpage}
\theoremstyle{plain}
\newtheorem{thm}{Theorem}[section]
\newtheorem{prob}[thm]{Problem}
\newtheorem{hyp}[thm]{Hypothesis}
\newtheorem{lem}[thm]{Lemma}
\newtheorem{cor}[thm]{Corollary}

\theoremstyle{definition}

\begin{document}
\title{Weierstrass Sigma Function\\ Coefficients Divisibility Hypothesis}
\author{Elena~Yu.~Bunkova}
\address{Steklov Mathematical Institute of Russian Academy of Sciences, Moscow, Russia}
\email{bunkova@mi.ras.ru}

\thanks{Supported in part by Young Russian Mathematics award and RFBR grant 16-51-55017.}

\maketitle

\begin{abstract}
We consider the coefficients in the series expansion at zero of the Weierstrass sigma function
\[
\sigma(z) = z \sum_{i, j \geqslant 0} {a_{i,j} \over (4 i + 6 j + 1)!} \left({g_2 z^4 \over 2}\right)^i \left(2 g_3 z^6\right)^j.
\]
We have $a_{i,j} \in \mathbb{Z}$. We present the divisibility Hypothesis for the integers $a_{i,j}$
\begin{align*}
\nu_2(a_{i,j}) &= \nu_2((4i + 6j + 1)!) - \nu_2(i!) - \nu_2(j!) - 3 i - 4 j, \\
\nu_3(a_{i,j}) &= \nu_3((4i + 6j + 1)!) - \nu_3(i!) - \nu_3(j!) - i - j.
\end{align*}
If this conjecture holds, then $\sigma(z)$ is a Hurwitz series over the ring $\mathbb{Z}[{g_2 \over 2}, 6 g_3]$.
\end{abstract}

\section{Exposition}

An \emph{elliptic function} is a meromorphic function on the complex torus $T = \mathbb{C}/\Gamma$,
where $\Gamma \subset \mathbb{C}$ is a lattice of rank $2$.
That is, an elliptic function is a meromorphic function in $\mathbb{C}$ with the periodicity property
\[
 f(z + \omega) = f(z) \quad \text{for} \quad \omega \in \Gamma.
\]

The \emph{Weierstrass sigma function} $\sigma(z) = \sigma(z; g_2, g_3)$
is an entire function in $\mathbb{C}$ determined by the system of equations
$Q_0 \sigma(z) = 0$, $Q_2 \sigma(z) = 0$, where
\begin{equation} \label{e1}
Q_0 = 4 g_2 {\partial \over \partial g_2} + 6 g_3 {\partial \over \partial g_3} - z {\partial \over \partial z} + 1, \quad
Q_2 = 6 g_3 {\partial \over \partial g_2} + {1 \over 3} g_2^2 {\partial \over \partial g_3}
- {1 \over 2} {\partial^2 \over \partial z^2} - {1 \over 24} g_2 z^2,
\end{equation}
and initial conditions $\sigma(0) = 0$, $\sigma'(0) = 1$. See \cite{Wei}.

The \emph{Weierstrass functions} $\zeta(z)$ and $\wp(z)$ are determined by the expressions
\[
\zeta(z)=\frac{\partial \ln \sigma(z)}{\partial z} \; \text{ and }\; \wp(z)=-\frac{\partial \zeta(z)}{\partial z}\,.
\]

Set $\Delta=g_2^3-27g_3^2$.
For $\Delta \ne 0$ the function $\wp(z)$ is elliptic. Denote it's lattice of periods by $\Gamma$.
Any elliptic function with periods $\Gamma$ is a rational function in
$\wp(z)$ and $\frac{\partial}{\partial z}\wp(z)$.

The functions $\wp(z)$ and $\wp'(z)$ satisfy the Weierstrass equation
\begin{equation} \label{curve}
 \wp'(z)^2 = 4 \wp(z)^3 - g_2 \wp(z) - g_3.
\end{equation}

\section{Weierstrass recursion}

Consider $\sigma(z)$ as a series expansion at $0$. Then $\sigma(z) \in \mathbb{Q}[g_2, g_3][[z]]$. From \eqref{e1} we have
\[
\sigma(z)=  z - {g_2 \over 2} {z^5 \over 5!} - 6 g_3 {z^7 \over 7!} -
9 {g_2^2 \over 4} {z^9 \over 9!} - 18 g_2 g_3 {z^{11} \over {11}!} + \ldots.
\]
Set
\[
\sigma(z) = z \sum_{i, j \geqslant 0} {a_{i,j} \over (4 i + 6 j + 1)!} \left({g_2 z^4 \over 2}\right)^i \left(2 g_3 z^6\right)^j, \quad a_{i,j} \in \mathbb{Q}.
\]

\begin{thm}[Weierstrass recursion, see \cite{Wei}] \label{WT} 
We have
\begin{equation} \label{Weireq}
a_{i,j} = 3 (i+1) a_{i+1,j-1} + {16 \over 3} (j+1) a_{i-2,j+1} - {1 \over 3} (4 i + 6 j -1) (2 i + 3 j - 1) a_{i-1,j}
\end{equation}
for $i \geqslant 0$, $j \geqslant 0$, $(i,j) \neq (0,0)$, and
\[
a_{0,0} = 1; \quad a_{i,j} = 0 \; \text{ for } \; i < 0 \text{ or } j < 0.
\]
\end{thm}

This theorem follows from \eqref{e1}. Let us note that it determines all the coefficients $a_{i,j}$.

\section{Hurwitz series expansion ring}

We say that $\varphi(z)$ is a Hurwitz series over the ring $R$ and write $\varphi(z) \in H R[[z]]$
if $\varphi(z) \in R \otimes \mathbb{Q}[[z]]$ and 
\[
 \varphi(z) = \sum_{k=0}^\infty \varphi_k {z^k \over k!}, \quad \varphi_k \in R.
\]

\begin{prob} \label{p1}
Find the minimal ring $\mathcal{R}$, such that $\sigma(z) \in H \mathcal{R}[[z]]$.
\end{prob}

We have $\mathbb{Z}[{g_2 \over 2}, 6 g_3] \subset \mathcal{R} \subset \mathbb{Q}[g_2, g_3]$.

\begin{cor}[from \ref{WT}] \label{c1}
$\sigma(z) \in H \mathbb{Z}[{1 \over 3}][{g_2 \over 2}, 2 g_3][[z]]$.
\end{cor}

\begin{thm}[\cite{O}] \label{c2}
 $\sigma(z) \in H \mathbb{Z}[{1 \over 2}][g_2, g_3][[z]]$.
\end{thm}
\begin{proof}
The article \cite{O} states that $\sigma(z) \in H \mathbb{Z}[{\mu_1 \over 2}, \mu_2, \mu_3, \mu_4, \mu_6][[z]]$
for the sigma-function of the general Weierstrass model of the elliptic curve
$y^2 + (\mu_1 x + \mu_3) y = x^3 + \mu_2 x^2 + \mu_4 x +\mu_6$.
Set $\mu_1 = \mu_2 = \mu_3 = 0$, $\mu_4 = - {g_2 \over 4}$, $\mu_6 = - {g_3 \over 4}$. 
We get the curve \eqref{curve}. Thus
$\sigma(z) \in H \mathbb{Z}[{g_2 \over 4},  {g_3 \over 4}][[z]]$ and therefore $\sigma(z) \in H \mathbb{Z}[{1 \over 2}][g_2, g_3][[z]]$.
\end{proof}

\begin{cor}[from \ref{c1}, \ref{c2}]
$\sigma(z) \in H \mathbb{Z}[{g_2 \over 2}, 2 g_3][[z]]$.
\end{cor}

Therefore in \eqref{Weireq} we have $a_{i,j} \in \mathbb{Z}$.

\section{Divisibility hypothesis}

Let $\nu_p(\cdot)$ denote p-adic valuation. Set 
\[
b_{i,j} = {(4i + 6j + 1)! \over 2^{3 i + 4 j} \, 3^{i + 2 j} \, i! \, j!}. 
\]

\begin{lem}
$b_{i,j} \in \mathbb{Z}$.
\end{lem}

\begin{proof}
As ${(4i + 6j + 1)! \over i! \, j! \, (3 i + 5 j + 1)!} \in \mathbb{Z}$ we have $\nu_p(b_{i,j}) \geqslant 0$ for prime $p >3$.

By Legendre's formula $\displaystyle\nu _{p}(n!) = \sum_{k=1}^\infty\Big\lfloor\frac{n}{p^k}\Big\rfloor$, so for $p = 2$ we have
\begin{multline*}
\nu _{2}((4i + 6j + 1)!) = \sum_{k=1}^\infty\Big\lfloor\frac{4i + 6j + 1}{2^k}\Big\rfloor
= 2 i + 3 j + \sum_{k=1}^\infty\Big\lfloor\frac{2i + 3j}{2^k}\Big\rfloor \geqslant \\ \geqslant
2 i + 3 j + i + j + \sum_{k=1}^\infty\Big\lfloor\frac{i}{2^k}\Big\rfloor + \sum_{k=1}^\infty\Big\lfloor\frac{j}{2^k}\Big\rfloor = 
3 i + 4 j + \nu _{2}(i!) + \nu _{2}(j!),
\end{multline*}
and for $p = 3$ we have
\begin{multline*}
\nu _{3}((4i + 6j + 1)!) = \sum_{k=1}^\infty\Big\lfloor\frac{4i + 6j + 1}{3^k}\Big\rfloor = i + 2 j + \Big\lfloor\frac{i + 1}{3}\Big\rfloor +
\sum_{k=1}^\infty\Big\lfloor\frac{i + 2 j + {i+1 \over 3}}{3^k}\Big\rfloor \geqslant \\ \geqslant
i + 2 j + \sum_{k=1}^\infty\Big\lfloor\frac{i}{3^k}\Big\rfloor + \sum_{k=1}^\infty\Big\lfloor\frac{j}{3^k}\Big\rfloor
= i + 2 j + \nu _{3}(i!) + \nu _{3}(j!).
\end{multline*}
\end{proof}

\begin{hyp} \label{hyppo}
$
\nu_2(a_{i,j}) = \nu_2(b_{i,j}), \quad \nu_3(a_{i,j}) = j + \nu_3(b_{i,j}). 
$
\end{hyp}

\begin{cor} If Hypothesis \ref{hyppo} holds,
the ring $\mathbb{Z}[{g_2 \over 2}, 6 g_3]$ solves problem \ref{p1}.
\end{cor}

\section{Numerical calculations}

The coefficients $a_{i,j}$ grow fast. We have, for example, the prime factorisation
\[
 a_{10,10} = - 2^{11} \cdot 3^{20} \cdot 5^2 \cdot 7 \cdot 11 \cdot 19 \cdot 415516114672128127554409484207124689335643.
\]

So far Hypothesis \ref{hyppo} has no proof. Numerical calculations using mathematical software (Maple 2015)
show that Hypothesis \ref{hyppo} holds for all $i \leqslant 100$, $j \leqslant 100$.

The following tables give $\nu_2(a_{i,j})$ for $0 \leqslant i \leqslant 20$, $0 \leqslant j \leqslant 20$ and
$\nu_3(a_{i,j})$, $\nu_3(a_{i,j}) - j$ and $\nu_5(a_{i,j})$ for $0 \leqslant i \leqslant 20$, $0 \leqslant j \leqslant 10$.

Table 4 suggests there is no formula like Hypothesis \ref{hyppo} for $\nu_p(a_{i,j})$ for prime $p>3$.

\text{ }

\text{ }

\begin{table}[h]
\caption{$\nu_2(a_{i,j})$}\label{t2}
\begin{center}
\begin{tabular}{|c|c|c|c|c|c|c|c|c|c|c|c|c|c|c|c|c|c|c|c|c|}
\hline
0 & 0 & 0 & 0 & 0 & 0 & 0 & 0 & 0 & 0 & 0 & 0 & 0 & 0 & 0 & 0 & 0 & 0 & 0 & 0 & 0 \\
\hline
0 & 1 & 0 & 2 & 0 & 1 & 0 & 3 & 0 & 1 & 0 & 2 & 0 & 1 & 0 & 4 & 0 & 1 & 0 & 2 & 0 \\
\hline
1 & 3 & 2 & 3 & 1 & 4 & 3 & 4 & 1 & 3 & 2 & 3 & 1 & 5 & 4 & 5 & 1 & 3 & 2 & 3 & 1 \\
\hline
3 & 3 & 3 & 3 & 4 & 4 & 4 & 4 & 3 & 3 & 3 & 3 & 5 & 5 & 5 & 5 & 3 & 3 & 3 & 3 & 4 \\
\hline
3 & 3 & 5 & 5 & 4 & 4 & 5 & 5 & 3 & 3 & 6 & 6 & 5 & 5 & 6 & 6 & 3 & 3 & 5 & 5 & 4 \\
\hline
3 & 6 & 5 & 6 & 4 & 6 & 5 & 6 & 3 & 7 & 6 & 7 & 5 & 7 & 6 & 7 & 3 & 6 & 5 & 6 & 4 \\
\hline
6 & 7 & 6 & 8 & 6 & 7 & 6 & 10& 7 & 8 & 7 & 9 & 7 & 8 & 7 & 10& 6 & 7 & 6 & 8 & 6 \\
\hline
7 & 7 & 8 & 8 & 7 & 7 & 10& 10& 8 & 8 & 9 & 9 & 8 & 8 & 10& 10& 7 & 7 & 8 & 8 & 7 \\
\hline
7 & 7 & 7 & 7 & 9 & 9 & 9 & 9 & 8 & 8 & 8 & 8 & 9 & 9 & 9 & 9 & 7 & 7 & 7 & 7 & 10\\
\hline
7 & 8 & 7 & 11& 9 & 10& 9 & 11& 8 & 9 & 8 & 11& 9 & 10& 9 & 11& 7 & 8 & 7 & 12& 10\\
\hline
8 & 12& 11& 12& 10& 12& 11& 12& 9 & 12& 11& 12& 10& 12& 11& 12& 8 & 13& 12& 13& 11\\
\hline
12& 12& 12& 12& 12& 12& 12& 12& 12& 12& 12& 12& 12& 12& 12& 12& 13& 13& 13& 13& 13\\
\hline
12& 12& 13& 13& 12& 12& 14& 14& 12& 12& 13& 13& 12& 12& 16& 16& 13& 13& 14& 14& 13\\
\hline
12& 14& 13& 14& 12& 15& 14& 15& 12& 14& 13& 14& 12& 17& 16& 17& 13& 15& 14& 15& 13\\
\hline
14& 15& 14& 17& 15& 16& 15& 17& 14& 15& 14& 19& 17& 18& 17& 19& 15& 16& 15& 18& 16\\
\hline
15& 15& 17& 17& 16& 16& 17& 17& 15& 15& 19& 19& 18& 18& 19& 19& 16& 16& 18& 18& 17\\
\hline
15& 15& 15& 15& 15& 15& 15& 15& 17& 17& 17& 17& 17& 17& 17& 17& 16& 16& 16& 16& 16\\
\hline
15& 16& 15& 17& 15& 16& 15& 20& 17& 18& 17& 19& 17& 18& 17& 20& 16& 17& 16& 18& 16\\
\hline
16& 18& 17& 18& 16& 21& 20& 21& 18& 20& 19& 20& 18& 21& 20& 21& 17& 19& 18& 19& 17\\
\hline
18& 18& 18& 18& 21& 21& 21& 21& 20& 20& 20& 20& 21& 21& 21& 21& 19& 19& 19& 19& 21\\
\hline
18& 18& 22& 22& 21& 21& 22& 22& 20& 20& 22& 22& 21& 21& 22& 22& 19& 19& 22& 22& 21\\
\hline
\end{tabular}
\end{center}
\end{table}

\begin{table}[h]
\caption{$\nu_3(a_{i,j})$}\label{t3}
\begin{center}
\begin{tabular}{|c|c|c|c|c|c|c|c|c|c|c|c|c|c|c|c|c|c|c|c|c|}
\hline
0 & 0 & 2 & 1 & 1 & 3 & 2 & 4 & 5 & 4 & 4 & 6 & 5 & 5 & 8 & 7 & 8 & 9 & 8 & 8 & 12\\
\hline
1 & 2 & 3 & 3 & 3 & 6 & 5 & 5 & 7 & 5 & 6 & 7 & 8 & 8 & 10& 9 & 9 & 11& 9 & 12& 13\\
\hline
3 & 3 & 5 & 4 & 6 & 7 & 7 & 7 & 9 & 7 & 7 & 10& 9 & 10& 11& 11& 11& 15& 13& 13& 15\\
\hline
4 & 4 & 7 & 6 & 6 & 8 & 7 & 8 & 9 & 9 & 9 & 11& 10& 10& 12& 11& 14& 15& 14& 14& 16\\
\hline
5 & 7 & 8 & 8 & 8 & 10& 9 & 9 & 12& 10& 11& 12& 12& 12& 16& 15& 15& 17& 15& 16& 17\\
\hline
8 & 8 & 10& 9 & 10& 11& 12& 12& 14& 12& 12& 14& 13& 16& 17& 17& 17& 19& 17& 17& 20\\
\hline
9 & 9 & 11& 10& 10& 13& 12& 13& 14& 13& 13& 17& 16& 16& 18& 17& 18& 19& 19& 19& 21\\
\hline
10& 11& 12& 13& 13& 15& 14& 14& 16& 14& 17& 18& 18& 18& 20& 19& 19& 22& 20& 21& 22\\
\hline
12& 12& 15& 14& 15& 16& 16& 16& 20& 18& 18& 20& 19& 20& 21& 22& 22& 24& 22& 22& 24\\
\hline
13& 13& 15& 14& 14& 16& 15& 18& 19& 18& 18& 20& 19& 19& 22& 21& 22& 23& 22& 22& 25\\
\hline
14& 15& 16& 16& 16& 20& 19& 19& 21& 19& 20& 21& 22& 22& 24& 23& 23& 25& 23& 25& 26\\
\hline
\end{tabular}
\end{center}
\end{table}

\begin{table}[h]
\caption{$\nu_3(a_{i,j}) - j$}\label{t3}
\begin{center}
\begin{tabular}{|c|c|c|c|c|c|c|c|c|c|c|c|c|c|c|c|c|c|c|c|c|}
\hline
\hspace{3pt}0\hspace{3pt} & \hspace{3pt}0\hspace{3pt} & \hspace{3pt}2\hspace{3pt} & \hspace{3pt}1\hspace{3pt} & \hspace{3pt}1\hspace{3pt}
& \hspace{3pt}3\hspace{3pt} & \hspace{3pt}2\hspace{3pt} & \hspace{3pt}4\hspace{3pt} & \,5\, & \,4\, & \,4\, & \,6\, & \,5\, & \,5\, & \,8\, & \,7\, & 8 & 9 & 8 & 8 & 12\\
\hline
0 & 1 & 2 & 2 & 2 & 5 & 4 & 4 & 6 & 4 & 5 & 6 & 7 & 7 & 9 & 8 & 8 & 10& 8 & 11& 12\\
\hline
1 & 1 & 3 & 2 & 4 & 5 & 5 & 5 & 7 & 5 & 5 & 8 & 7 & 8 & 9 & 9 & 9 & 13& 11& 11& 13\\
\hline
1 & 1 & 4 & 3 & 3 & 5 & 4 & 5 & 6 & 6 & 6 & 8 & 7 & 7 & 9 & 8 & 11& 12& 11& 11& 13\\
\hline
1 & 3 & 4 & 4 & 4 & 6 & 5 & 5 & 8 & 6 & 7 & 8 & 8 & 8 & 12& 11& 11& 13& 11& 12& 13\\
\hline
3 & 3 & 5 & 4 & 5 & 6 & 7 & 7 & 9 & 7 & 7 & 9 & 8 & 11& 12& 12& 12& 14& 12& 12& 15\\
\hline
3 & 3 & 5 & 4 & 4 & 7 & 6 & 7 & 8 & 7 & 7 & 11& 10& 10& 12& 11& 12& 13& 13& 13& 15\\
\hline
3 & 4 & 5 & 6 & 6 & 8 & 7 & 7 & 9 & 7 & 10& 11& 11& 11& 13& 12& 12& 15& 13& 14& 15\\
\hline
4 & 4 & 7 & 6 & 7 & 8 & 8 & 8 & 12& 10& 10& 12& 11& 12& 13& 14& 14& 16& 14& 14& 16\\
\hline
4 & 4 & 6 & 5 & 5 & 7 & 6 & 9 & 10& 9 & 9 & 11& 10& 10& 13& 12& 13& 14& 13& 13& 16\\
\hline
4 & 5 & 6 & 6 & 6 & 10& 9 & 9 & 11& 9 & 10& 11& 12& 12& 14& 13& 13& 15& 13& 15& 16\\
\hline
\end{tabular}
\end{center}
\end{table}

\begin{table}[h]
\caption{$\nu_5(a_{i,j})$}\label{t5}
\begin{center}
\begin{tabular}{|c|c|c|c|c|c|c|c|c|c|c|c|c|c|c|c|c|c|c|c|c|}
\hline
\hspace{3pt}0\hspace{3pt} & \hspace{3pt}0\hspace{3pt} & \hspace{3pt}0\hspace{3pt} & \hspace{3pt}0\hspace{3pt}
& \hspace{3pt}0\hspace{3pt} & \hspace{3pt}0\hspace{3pt} & \hspace{3pt}0\hspace{3pt} & \hspace{3pt}0\hspace{3pt}
& \hspace{3pt}0\hspace{3pt} & \hspace{3pt}0\hspace{3pt} & \hspace{3pt}0\hspace{3pt} & \hspace{3pt}0\hspace{3pt}
& \hspace{3pt}0\hspace{3pt} & \hspace{3pt}0\hspace{3pt} & \hspace{3pt}0\hspace{3pt} & \hspace{3pt}0\hspace{3pt}
& \hspace{3pt}0\hspace{3pt} & \hspace{3pt}0\hspace{3pt} & \hspace{3pt}0\hspace{3pt} & \hspace{3pt}0\hspace{3pt} & \hspace{3pt}0\hspace{3pt} \\
\hline
0 & 0 & 0 & 0 & 1 & 0 & 0 & 0 & 0 & 1 & 0 & 0 & 0 & 0 & 1 & 0 & 0 & 0 & 0 & 1 & 0 \\
\hline
0 & 0 & 1 & 1 & 1 & 0 & 0 & 1 & 1 & 1 & 0 & 0 & 1 & 1 & 1 & 0 & 0 & 1 & 1 & 1 & 0 \\
\hline
0 & 2 & 1 & 1 & 1 & 0 & 1 & 1 & 1 & 1 & 0 & 1 & 1 & 1 & 1 & 0 & 1 & 1 & 1 & 1 & 0 \\
\hline
2 & 1 & 1 & 1 & 1 & 2 & 1 & 1 & 1 & 1 & 2 & 1 & 1 & 1 & 1 & 2 & 1 & 1 & 1 & 1 & 2 \\
\hline
1 & 1 & 1 & 1 & 1 & 1 & 1 & 1 & 1 & 1 & 2 & 1 & 1 & 2 & 1 & 1 & 1 & 2 & 1 & 2 & 2\\
\hline
1 & 1 & 1 & 1 & 2 & 1 & 2 & 1 & 1 & 2 & 1 & 1 & 1 & 1 & 3 & 1 & 1 & 2 & 2 & 2 & 3\\
\hline
1 & 1 & 3 & 2 & 3 & 3 & 1 & 2 & 2 & 2 & 1 & 1 & 2 & 3 & 2 & 1 & 2 & 2 & 3 & 3 & 2\\
\hline
1 & 3 & 2 & 2 & 2 & 1 & 2 & 2 & 2 & 3 & 1 & 2 & 5 & 2 & 2 & 2 & 2 & 4 & 3 & 3 & 2\\
\hline
3 & 2 & 2 & 2 & 2 & 3 & 2 & 2 & 3 & 2 & 3 & 3 & 2 & 2 & 2 & 3 & 3 & 3 & 3 & 4 & 4\\
\hline
2 & 2 & 2 & 2 & 3 & 2 & 2 & 2 & 3 & 2 & 2 & 5 & 2 & 2 & 4 & 3 & 3 & 4 & 3 & 3 & 3\\
\hline
\end{tabular}
\end{center}
\end{table}

\text{ }

\text{ }

\end{document}